\documentclass[reqno,12pt]{amsart}
\usepackage{bbm}
\usepackage[all,pdf]{xy}
\usepackage{epsfig}
\usepackage{amsmath}
\usepackage{amssymb}
\usepackage{amscd}
\usepackage{graphicx}
\usepackage{amsmath}
\allowdisplaybreaks[4]
\makeatletter
\@namedef{subjclassname@2020}{%
	\textup{2020} Mathematics Subject Classification}
\makeatother
\usepackage[colorlinks=true]{hyperref}
\usepackage{amsfonts}
\usepackage[top=35mm, bottom=35mm, left=30mm, right=30mm]{geometry}

\newtheorem{thm}{Theorem}[section]
\newtheorem{lemma}[thm]{Lemma}

\newtheorem{fact}{Fact}
\newtheorem{ques}[thm]{Question}
\newtheorem{cor}[thm]{Corollary}
\newtheorem{defn}[thm]{Definition}
\newtheorem{rem}[thm]{Remark}
\newtheorem{step}{Step}

\def \N {\mathbb N}

\def \M {\mathcal M}

\def\B {\mathcal B}
\def\lk {\left}
\def\re {\right}
\numberwithin{equation}{section}

\begin{document}

	\baselineskip 14pt
	
	\title[Pinsker  $\sigma$-algebra character and mean Li-Yorke chaos]{Pinsker  $\sigma$-algebra character and mean Li-Yorke chaos}
	

	\author{Chunlin Liu}
	\address[Chunlin Liu]{School of Mathematical Sciences, University of Science and Technology of China, Hefei, Anhui, 230026, P.R. China}
	\email{lcl666@mail.ustc.edu.cn}
	
	\author{Rongzhong Xiao} 
	\address[Rongzhong Xiao]{School of Mathematical Sciences, University of Science and Technology of China, Hefei, Anhui, 230026, P.R. China}
	 \email{xiaorz@mail.ustc.edu.cn}
	 
	\author{Leiye Xu}	
	\address[Leiye Xu]{School of Mathematical Sciences, University of Science and Technology of China, Hefei, Anhui, 230026, P.R. China}
\email{leoasa@mail.ustc.edu.cn}


 \begin{abstract}
 Let $G$ be an infinite countable discrete amenable group. For any $G$-action on a compact metric space $X$, it is proved that for any sequence $(G_n)_{n\ge 1}$ consisting of non-empty finite subsets  of $G$ with $\lim_{n\to \infty}|G_n|=\infty$, Pinsker $\sigma$-algebra is a characteristic factor for $(G_n)_{n\ge 1}$. As a consequence, for a class of $G$-topological dynamical systems, positive topological entropy	implies mean Li-Yorke chaos along a class of sequences consisting of non-empty finite subsets  of $G$.
	\end{abstract}
		\maketitle
		
	\section{Introduction}

	 Throughout this paper, $G$ is an infinite countable discrete amenable group. A {\em $G$-topological dynamical system} ($G$-tds for short) $(X, G)$ is such that  $X$ is a compact metric space endowed with the metric $d$ and $G$ acts on $X$ continuously. Let $\mathcal{B}_X $ be the Borel $\sigma$-algbra of $X$. A {\em $G$-measure preserving dynamical system}  ($G$-mps for short) $(X,\mathcal{B}_X^\mu,\mu,G)$  is such that $(X,G)$ is a $G$-tds,  $\mu$ is an invariant  Borel probability measure of $(X,G)$ and $\mathcal{B}_{X}^{\mu}$ is the completion of $\mathcal{B}_{X}$ under $\mu$.
	 
	  For a non-empty set $S$, we let $F(S)$ be the set of all non-empty finite subsets of $S$. By
	 a  probability space  $(X,\mathcal{B}_X^\mu,\mu)$, we let
	 $L^{2}(X,\B_{X}^\mu,\mu)$ the set of real-valued measurable functions $f:X\to\mathbb{R}$ with
	 $\|f\|_2=\left(\int f^2d\mu\right)^{\frac{1}{2}}<\infty.$ For a $G$-mps $(X,\mathcal{B}_X^\mu,\mu,G)$, we use the symbol $\textbf{A}(F,f)(x)$ to denote the average
	 $$\textbf{A}(F,f)(x)=\frac{1}{|F|}\sum_{g\in F}f(gx)$$
	 for any  $f:X\to\mathbb{R}$ and any $F\in F(G)$. 
	 
	 In \cite{FW96}, Furstenberg and Weiss introduced the notion of characteristic factor when they consider the
	 $L^{2}$-limit of some polynomial ergodic averages. In 2005, Host and Kra characterized the characteristic factor of non-conventional ergodic averages via so-called pro-nilfactor when they study $L^{2}$-limit of non-conventional ergodic averages (see \cite{book:HK} for more details). There is a natural question:
	 \begin{ques}\label{q1}
    Let  $(X,\mathcal{B}_X^\mu,\mu,G)$ be an ergodic $G$-mps and $(G_n)_{n\ge 1}\subset F(G)$ with $\lim_{n\to\infty}|G_n|=\infty$. Can we find the characteristic factor when we consider the $L^{2}$-convergence of ergodic averages $\textbf{A}(G_n,f)$ for any $f\in L^2(X,\mathcal{B}_X^\mu,\mu)$?
	 \end{ques}
 In this paper, we find such a characteristic factor. Let us begin with the definition of the characteristic factor.
	 \begin{defn}Let  $(X,\mathcal{B}_X^\mu,\mu,G)$ be a $G$-mps and $(G_n)_{n\ge 1}\subset F(G)$. A sub-$\sigma$-algebra $\mathcal{F}$ of $\mathcal{B}_X^\mu$ is called a characteristic factor for $(G_n)_{n\ge 1}$ if for every $f\in L^2(X,\mathcal{B}_X^\mu,\mu)$,
	 	$$\lim_{n\to\infty}\left\|\textbf{A}(G_n,f)-\textbf{A}(G_n,\mathbb{E}(f|\mathcal{F}))\right\|_2=0,$$
	 	where $\mathbb{E}(f|\mathcal{F})$ is the conditional expectation of $f$ with respect to $\mathcal{F}$.
	 \end{defn}
The following result shows that Pinsker $\sigma$-algebra suits Question \ref{q1}.
	 \begin{thm}\label{Thm-A}
	 	Let $(X,\mathcal{B}_X^\mu,\mu, G)$ be an ergodic $G$-mps, $(G_n)_{n\ge 1}\subset F(G)$ and $\mathcal{P}_\mu(G)$  be the Pinsker $\sigma$-algebra of $(X,\mathcal{B}_X^\mu,\mu, G)$. If $\lim_{n\to\infty}|G_n|=\infty$, then $\mathcal{P}_\mu(G)$ is a characteristic factor  for $(G_n)_{n\ge 1}$.
	 \end{thm}
A  $G$-mps is a  Kolmogorov system if its Pinsker $\sigma$-algebra is trivial. The following is immediately from  Theorem \ref{Thm-A}.
\begin{cor}Let $(X,\mathcal{B}_X^\mu,\mu, G)$ be a Kolmogorov $G$-mps and  $(G_n)_{n\ge 1}\subset F(G)$. If $\lim_{n\to\infty}|G_n|=\infty$ then for all $f\in L^2(X,\mathcal{B}_X^\mu,\mu)$,
	$$\lim_{n\to\infty}\left\|\textbf{A}(G_n,f)-\int fd\mu\right\|_2=0.$$
\end{cor}
As an application of Theorem \ref{Thm-A}, we study the mean Li-Yorke chaos.  In the topological dynamics, a fundamental question is to find the relationship between positive topological entropy and chaos. Using the ergodic theoretic method,
	Blanchard, Glasner, Kolyada and Maass \cite{BGKM} proved that for a $\mathbb{Z}$-tds
	positive topological entropy implies Li-Yorke chaos.
	 Downarowicz \cite{D} observed that for a $\mathbb{Z}$-tds,  mean Li-Yorke chaos
	is equivalent to so-called DC2 chaos
	and proved that positive topological entropy implies mean Li-Yorke chaos. In
	\cite{HLY14}, 
	Huang, Li and Ye provided a different approach and
	showed that positive topological entropy implies a
	multivariant version of mean Li-Yorke chaos. By using combinatorial methods, Kerr and Li proved that positive topological entropy implies Li-Yorke chaos
	for amenable
	group actions in \cite{KL07}, and for sofic group actions in \cite{KL13}. For more related topics on chaotic behaviours
	we refer to \cite{GJ,HLY-Arxiv,HJ,HXY,HY02,I,LR, LY, WG}.

With the deepening of research, scholars began to pay attention to the sequence version of the Li-Yorke chaos. Huang, Li and Ye \cite{HLY-Arxiv} showed for any $G$-tds, positive topological entropy implies the Li-Yorke chaos along any sequence with pairwise distinct elements. Li and Qiao \cite{LQ} shown that for a $\mathbb{Z}$-tds and a ``good sequence'' of $\N$, positive topological entropy implies mean Li-Yorke chaos along this good sequence.
However, the condition ``good sequence'' is not valid for all sequences even in $\mathbb{N}$. In this paper, by argument of weak convergence, we can omit  the condition.  An advantage of our method is to avoid using some ergodic theorems when one studies sequence version of mean Li-Yorke chaos.
	 	
To be precise, let $(X,G)$ be a $G$-tds.	For $(G_n)_{n\ge 1}\subset F(G)$ and $\delta>0$,
	a pair $(x_1,x_2)\in X \times X$ is called a $((G_n)_{n\ge 1},\delta)$-mean scrambled pair if 
	\begin{equation*}\label{100}
		\liminf_{n\rightarrow \infty}\frac{1}{|G_n|}\sum_{g\in G_n}d(gx_1,gx_2)=0
	\end{equation*}
	and 
	\begin{equation*}\label{200}
		\limsup_{n\rightarrow \infty}\frac{1}{|G_n|}\sum_{g\in G_n}d(gx_1,gx_2) >\delta.
	\end{equation*}
	A subset $K$ of $X$ with at least two points is called a $((G_n)_{n\ge 1},\delta)$-mean scrambled subset of $X$ if every two  distinct points $x_1,x_2\in K$ form a  $((G_n)_{n\ge 1},\delta)$-mean scrambled pair. Conventionally,  $(X,G)$ is said to be   $((G_n)_{n\ge 1},\delta)$-mean Li-Yorke chaotic if there exists an uncountable $((G_n)_{n\ge 1},\delta)$-mean scrambled subset of $X$.

	Let $(X,G)$ be a $G$-tds. For $(G_n)_{n\ge 1}\subset F(G)$ and $\delta>0$, put 
$$W_{(G_n)_{n\ge 1},\delta}=\{(x_1,x_2)\in X\times X:\limsup_{n\rightarrow \infty}\frac{1}{|G_n|}\sum_{g\in G_n}d(gx_1,gx_2) >\delta\}.$$
By the arguement of the weak convergence, we obtain the following theorem. 
\begin{thm}\label{thm1}
	Let $(X,G)$ be a $G$-tds, $(G_n)_{n\ge 1}\subset F(G)$ and $\mu$  be an ergodic measure of $(X,G)$. Assume $h_\mu(G)>0$ and $\lim_{n\to\infty}|G_n|=\infty$. Let  $\mu=\int \mu_yd\mu(y)$ be the disintegration of $\mu$ with respect to $\mathcal{P}_{\mu}(G)$. Then
	for $\mu$-a.e. $x\in X$, there is $\delta_x>0$ such that  $\mu_x\times\mu_x(W_{(G_n)_{n\ge 1},\delta_x})=1$.
\end{thm}
Let $(X,G)$ be a $G$-tds. For  $(G_n)_{n\ge 1}\subset F(G)$ and $x\in X$, put 
$$W_{S,(G_n)_{n\ge 1}}(x)=\{y\in X:\lim_{n\rightarrow \infty}\frac{1}{|G_n|}\sum_{g\in G_n}d(gx,gy)=0\}.$$ 
By using Theorem \ref{thm1}, we can prove the following result.
\begin{thm}\label{thm2}
Under the assumption in Theorem \ref{thm1}, if 
\begin{align}\label{Condition}\overline{\text{supp}(\mu_x)\cap W_{S,(G_n)_{n\ge 1}}(x) }=\text{supp}(\mu_x)\end{align}
for $\mu$-a.e. $x\in X$, then for $\mu$-a.e. $x\in X$
there are $\delta_x>0$ and a dense Mycielski\footnote{For definition, see section 2.2.}  $((G_n)_{n\ge 1},\delta_x)$-mean scrambled subset $K_x$ of $supp(\mu_x)$.  \end{thm} 
Note that \eqref{Condition} is satisfied for all $\mathbb{Z}$-tds with positive topological entropy and any sequence $(G_n)_{n\ge 1}\subset F(\N)$ with $\lim_{n\to\infty}|G_n|=\infty$.

Moreover, we introduce a class of $G$-tds and a class of sequences consisting of non-empty finite subsets  of $G$ which satisfy \eqref{Condition}. By a group $G$, a subset $\Phi$ of $G$ is said to be an algebraic past of $G$ if \begin{itemize}
	\item[(1)]$\Phi \cap {\Phi}^{-1}=\emptyset$;
	\item[(2)]$\Phi \cup {\Phi}^{-1} \cup \{e_G\}=G$;
	\item[(3)]$\Phi \cdot \Phi \subset \Phi$.
\end{itemize}
The group $G$ is left-orderable if there exists a linear ordering in $G$ which is invariant under left translation. The group $G$ is left-orderable if and only if there exists an algebraic past $\Phi$ in $G$ \cite{S}. Indeed, one can obtain the desired linear order based on $\Phi$ as follows: $g_1< g_2$ if ${g_2}^{-1}g_1\in \Phi$. An  subset $S$ of $G$ is said to be an infinite subset of $G$ with respect to $\Phi$ if $|S|=\infty$ and
$|\{g<h:g\in S\}|<\infty$
for all $h\in G$.

Let  $G$ be a left-orderable amenable group, $\Phi$  be an  algebraic past of  $G$ with $g\Phi g^{-1}\subset\Phi$ for all $g\in G$, $(X,G)$ be a $G$-tds, $\mu$  be an ergodic measure of $(X,G)$ with $h_\mu(G)>0$, $S$  be an infinite subset of $G$ with respect to $\Phi$. Under the previous assumptions, Huang and Jin \cite[Claim 3 in Page 9]{HJ} prove that condition \eqref{Condition} holds for $(G_n)_{n\ge1}\subset F(S)$ with $\lim_{n\to\infty}|G_n|=\infty$. 
Then the following is immediately from  Theorem \ref{thm2}.
\begin{thm}\label{30}
	Let  $G$ be an infinite countable discrete left-orderable amenable group, $\Phi$  be an  algebraic past of  $G$ with $g\Phi g^{-1}\subset\Phi$ for all $g\in G$, $(X,G)$ be a $G$-tds, $\mu$  be an ergodic measure of $(X,G)$ with $h_\mu(G)>0$, $S$  be an infinite subset of $G$ with respect to $\Phi$. Assume $(G_n)_{n\ge 1}\subset F(S)$ satisfies $\lim_{n\to\infty}|G_n|=\infty$. Let  $\mu=\int \mu_yd\mu(y)$ be the disintegration of $\mu$ with respect to $\mathcal{P}_{\mu}(G)$.  Then
	for $\mu$-a.e. $x\in X$, there is $\delta_x>0$ and a dense Mycielski  $((G_n)_{n\ge 1},\delta_x)$-mean scrambled subset $K_x$ of $supp(\mu_x)$.
\end{thm}

\begin{rem}
	a. After finishing the article and submitting it to a journal, the referee from the journal told us that Theorem \ref{Thm-A} is a direct consequence of \cite[Theorem 5.4]{DG}. But our method is different from one of \cite{DG}. The key points of our method are definition of measure-theoretic entropy and Theorem \ref{6}. Based on the difference and completeness of this article, we still retain the proof of Theorem \ref{Thm-A}.
	
	b. The results of this article on mean Li-Yorke chaos, including Theorem \ref{thm1},\ref{thm2} and \ref{30}, are new. And Theorem \ref{thm2} is new even when $G=\mathbb{Z}$.
\end{rem}
	
	This paper is organized as follows. In Section 2, we review some notions and required results. In Section 3, we prove Theorem \ref{Thm-A}. In Section 4, we prove Theorem \ref{thm1} and Theorem \ref{thm2}. 
	
	\section{Preliminaries}
	
	In this section we recall some basic notions and results of $G$-tds. 
	
	\subsection{F\o lner sequence and entropy}
Let $G$ be a countable, discrete, infinite, amenable group. A sequence $\left\{F_{n}\right\}_{n=1}^{+\infty}$ of non-empty finite subsets of $G$ is called a F\o lner sequence if for every $g \in G$,
$$
\lim _{n \rightarrow+\infty} \frac{\left|g F_{n} \Delta F_{n}\right|}{\left|F_{n}\right|}=0,
$$
where $|\cdot|$ denotes the cardinality of a set. It is well known that $G$ is amenable if and only if it admits a F\o lner sequence.

	Let $(X,G)$ be a $G$-tds, where $G$ is a countable, discrete, infinite, amenable group. A cover of $X$ is a finite family of subsets of $X$ whose union is $X .$ A partition of $X$ is a cover of $X$ whose elements are pairwise disjoint. Denote by $\mathcal{C}_{X}$ (resp. $\left.\mathcal{C}_{X}^{o}\right)$ the set of all open covers (resp. finite open covers) of $X$ and by $\mathcal{P}_{X}$ (resp. $\left.\mathcal{P}_{X}^{b}\right)$ the set of all partitions of $X$ (resp. finite Borel partitions).
	Given two covers $\mathcal{U}, \mathcal{V} \in \mathcal{C}_{X}$, $\mathcal{U}$ is said to be finer than $\mathcal{V}$ (denoted by $\left.\mathcal{U} \succeq \mathcal{V}\right)$ if each element of $\mathcal{U}$ is contained in some element of $\mathcal{V}$. Let $\mathcal{U} \vee \mathcal{V}=\{U \cap V: U \in \mathcal{U}, V \in \mathcal{V}\}$. Denote by $N(\mathcal{U})$ the number of sets in a subcover of $\mathcal{U}$ of minimal cardinality. 
	
	The entropy of $\mathcal{U} \in \mathcal{C}_{X}^{o}$ with respect to $G$-action is defined by
	$$
	h_{t o p}(G, \mathcal{U})=\lim _{n \rightarrow+\infty} \frac{1}{\left|F_{n}\right|} \log N\left(\bigvee_{g \in F_{n}} g^{-1} \mathcal{U}\right)
	$$
	where $\{F_{n}\}_{n=1}^\infty$ is a F\o lner sequence of the group amenable $G.$ As is shown in \cite[Theorem 6.1]{L} the limit exists and is independent of the choice of F\o lner sequences. The topological entropy of $(X, G)$ is defined by
	$$
	h_{t o p}(G)=h_{t o p}(G, X)=\sup _{\mathcal{U} \in C_{X}^{o}} h_{t o p}(G, \mathcal{U}).
	$$

	Denote by $\mathcal{B}_{X}$ the collection of all Borel subsets of $X$ and $\mathcal{M}(X)$ the set of all Borel probability measures on $X$. For $\mu \in \mathcal{M}(X)$, denote by $\operatorname{supp}(\mu)$ the support of $\mu$, i.e., the smallest closed subset $W \subseteq X$ such that $\mu(W)=1 .$ $\mu \in \mathcal{M}(X)$ is called $G$-invariant if $g \mu=\mu$ for each $g \in G$, and called ergodic if it is $G$-invariant and $\mu\left(\bigcup_{g \in G} g A\right)=0$ or 1 for any $A \in \mathcal{B}_{X}.$ Denote by $\mathcal{M}(X, G)$ (resp. $\left.\mathcal{M}^{e}(X, G)\right)$ the set of all $G$-invariant (resp. ergodic) elements in $\mathcal{M}(X)$. Note that the amenability of $G$ ensures that $\mathcal{M}^{e}(X, G) \neq \emptyset$ and both $\mathcal{M}(X)$ and $\mathcal{M}(X, G)$ are convex compact metric spaces when they are endowed with the weak*-topology.

	Given a finite measurable parition $\alpha$ of $X$ and a sub-$\sigma$-algebra $\mathcal{A}$ of $\mathcal{B}_{X}^{\mu}$, define
	$$
	H_{\mu}(\alpha \mid \mathcal{A})=\sum_{A \in \alpha} \int_{X}-\mathbb{E}\left(1_{A} \mid \mathcal{A}\right) \log \mathbb{E}\left(1_{A} \mid \mathcal{A}\right) d \mu
	$$
	where $\mathbb{E}\left(1_{A} \mid \mathcal{A}\right)$ is the expectation of $1_{A}$ with respect to $\mathcal{A} .$ One standard fact is that $H_{\mu}(\alpha \mid \mathcal{A})$ increases with respect to $\alpha$ and decreases with respect to $\mathcal{A} .$ If $ \mathcal{A}=\{\emptyset, X\}$ then we have
	$$
	H_{\mu}(\alpha)=H_{\mu}(\alpha \mid \mathcal{A})=\sum_{A \in \alpha}-\mu(A) \log \mu(A).
	$$
	Given $\mu \in \mathcal{M}(X, G)$ and $\alpha \in \mathcal{P}_{X}^{b},$ the measure-theoretic entropy of $\mu$ relative to $\alpha$ is defined by
	$$
	h_{\mu}(G, \alpha)=\lim _{n \rightarrow+\infty} \frac{1}{\left|F_{n}\right|} H_{\mu}\left(\bigvee_{g \in F_{n}} g^{-1} \alpha\right)
	$$
	where $\{F_{n}\}_{n=1}^\infty$ is a Fl\o lner sequence of the amenable group $G$. As shown in \cite[Theorem 6.1]{L}, the limit exists and is independent of F\o lner sequences. The measure-theoretic entropy of $\mu$ is defined by
	$$
	h_{\mu}(G)=h_{\mu}(G, X)=\sup _{\alpha \in \mathcal{P}_{X}^{b}} h_{\mu}(G, \alpha)
	$$
	

	\subsection{Mycielski's theorem}
Let $X$ be a compact metric space. A subset $K \subset X$ is called a Mycielski set if it is a union of countable many Cantor sets, which was introduced in \cite{BGKM}. For $n\ge 2$, let  $\Delta^{(n)}=\{(x_1, x_2, \cdots, x_n)\in X^n:  x_i=x_j \text{ for some }1\le i< j\le n \}$. The following result is from \cite[Theorem 1]{M}.
	
	\begin{thm}\label{Mycielski}
Suppose that $X$ is a perfect compact metric space. If for each $n \ge 2$, $R_n$ is a dense $G_\delta$ subset of $X^n$, then there exists a dense Mycielski subset $K$ of  $X$ such that $K^n\subset R_n \cup \Delta^{(n)}$ holds for all integers $n\ge 2$. 
	\end{thm}

	\subsection{Disintegration of measures}
	Let $X$ be a compact metric space, $\mathcal{B}_X $ the Borel $\sigma$-algbra of $X$
	and $\mu$ a  Borel probability measure on $X$.  Let $\mathcal{B}_{X}^{\mu}$ be the completion of $\mathcal{B}_{X}$ under $\mu$. Let $\mathcal{F}$ be a sub-$\sigma$-algebra of $\B_{X}^\mu$.  Then there exists  a measurable map $X\rightarrow \mathcal{M}(X)$, $y\mapsto \mu_y$ such that for every $f\in L^1(X,\B_{X}^\mu,\mu)$, $\mathbb{E}(f|\mathcal{F})(y)=\int fd\mu_y$ for $\mu$-a.e. $y\in X$.
	
	The following lemma is from \cite{R} (see Lemma 3 in Section 4 No. 2).
	
	\begin{lemma} \label{2}Let $\mu$ be a probability measure on $X$ and $\mathcal{F}$ be a sub-$\sigma$-algebra of $\B_{X}^\mu$.
	Let $\mu=\int \mu_yd\mu(y)$ be the disintegration of $\mu$ with respect to $\mathcal{F}$. Suppose $\mu_y$ is non-atomic for $\mu$-a.e. $y\in X$. If $0\le r\le 1$ and $A\in \B_{X}^\mu$ with $\mu_y(A)\le r$ for $\mu$-a.e. $y\in X$, then there exists $A'\in \B_{X}^\mu$ such that $A\subset A'$ and $\mu_y(A')=r$ for $\mu$-a.e. $y\in X$.
	\end{lemma}

%

	The following result gives a decomposition of positive measurable sets.
	\begin{lemma}\label{lemma2}Let $\mu$ be a probability measure on $X$ and $\mathcal{F}$ be a sub-$\sigma$-algebra of $\B_{X}^\mu$.
		Let $\mu=\int \mu_yd\mu(y)$ be the disintegration of $\mu$ with respect to $\mathcal{F}$.  Suppose $\mu_y$ is non-atomic for $\mu$-a.e. $y\in X$. Then for any $A\in \B_{X}^\mu$ with $\mu(A)>0$ and for any $\epsilon>0$, there exist finitely many disjoint measurable sets $A_1,A_2,\cdots,A_k$ such that 
		\begin{itemize}
		\item 	$\mu(A\triangle \bigcup_{i=1}^{k}A_i)<\epsilon$.
		\item For each $1\le i \le k$, we can find $B_i\in \mathcal{F}$ with $\mu(B_i)>0$ and $q_i>0$ satisfying that for $\mu$-a.e. $y\in B_i$, $\mu_{y}(A_i)=q_i$ and for $\mu$-a.e. $y\in B_i^{c}$, $\mu_{y}(A_i)=0$. 
		\end{itemize}
	\end{lemma}
	
	\begin{proof}
		For any $\epsilon>0$, there exists $n\in \N$ such that ${1}/{n}<\epsilon$. Let $r_i={i}/{n}$ and $B_i=\{y\in X:r_i< \mu_{y}(A)\le r_i+{1}/{n}\}\cap \{y\in X:\mu_{y}(\alpha(y))=1\}$ where $\alpha(y)$ is an atom containing $y$ of $\mathcal{F}$, $i=0,1,\cdots,n-1$. Note that for $\mu$-a.e. $y\in B_i$, 
		$$r_i<\mu_{y}(A\cap B_i)\le r_i+{1}/{n}$$
		for each $i=0,1,\ldots,n-1.$
		 By Lemma \ref{2}, for each $i=0,1,\ldots,n-1$ there exists measurable subset $A_i$ of $X$ such that $A\cap B_i\subset A_i \subset B_i$ and for $\mu$-a.e. $y\in B_i$, $\mu_{y}(A_i)=r_i+{1}/{n}$. Then 
		 \begin{align*}
		 \mu(A\triangle \bigcup_{i=1}^{k}A_i)&=\sum_{i=0}^{n-1}\int_{B_i}\mu_{y}(A\triangle A_i)d\mu(y)\\
		 &=\sum_{i=0}^{n-1}\int_{B_i}\mu_{y}((A\cap B_i)\triangle A_i)d\mu(y)\\
		 &=\sum_{i=0}^{n-1}\int_{B_i}\mu_{y}(A_i\setminus(A\cap B_i))d\mu(y)\\
		 &\le \frac{1}{n}.
		 \end{align*}
			Hence the above  $A_1,A_2,\ldots, A_k$ can reach our requirements. 
	\end{proof}
	
	\subsection{Pinsker $\sigma$-algebra}
	
	Let $(X,\B_{X}^\mu,\mu,G)$ be a $G$-mps. The Pinsker $\sigma$-algebra of $(X,\B_{X}^\mu,\mu,G)$ is defined by $$\mathcal{P}_{\mu}(G)=\{A\in\mathcal{B}_X^\mu:h_{\mu}(G,\{A,A^c\})=0\}.$$ 
	
	Let $\mathcal{C}$ be a sub-$\sigma$-algebra of $\mathcal{B}^{\mu}_X$. The {\it relatively independent self-joining of $\mu$ with respect to $\mathcal{C}$} is the Borel probability measure $$\mu\times_\mathcal{C} \mu =\int\mu_y\times\mu_yd\mu(y)$$ on $X\times X$  in the sense that
	$$\mu\times_\mathcal{C} \mu(A\times B)=\int\mu_{y}(A)\times\mu_y(B)d\mu(y)$$
	for all $A,B \in \mathcal{B}_X$ where $\mu=\int \mu_yd\mu(y)$ is the disintegration of $\mu$ with respect to $\mathcal{C}$.
	
		The following result is well-known (for example, see \cite[ Proposition 3.1]{WG}).
	
	\begin{lemma} \label{3}
		Let $(X,G)$ be a $G$-tds and $\mu \in \mathcal{M}^{e}(X, G)$ with $h_{\mu}(G)>0$. Let $\mu=\int \mu_yd\mu(y)$ be the disintegration of $\mu$ with respect to  $\mathcal{P}_\mu(G)$ and $\lambda=\mu \times_{\mathcal{P}_\mu(G)} \mu$. Then $\mu_{y}$ is non-atomic for $\mu$-a.e. $y\in X$ and $\lambda \in \mathcal{M}^{e}(X \times X, G)$.
	\end{lemma}

The following result plays a key role in our proof of the main result. It was
proved in \cite[Theorem 2.13]{RW}, \cite[Theorem 0.1]{DD} and \cite[Theorem 6.10]{HYZ}.
	
	\begin{thm} \label{6}
		Let $(X,G)$ be a $G$-tds and $\mu \in \mathcal{M}^{e}(X, G)$. Let $\alpha$ be a finite measurable partition of $X$ and $\epsilon>0$. Then there exists a finite subset $K$ of $G$ such that for every finite subset $Q$ of $G$ with $(QQ^{-1} \backslash \{e_G\})\cap K=\emptyset$, one has 
		\begin{equation}\label{4}
			\left|H_{\mu}(\alpha | \mathcal{P}_{\mu}(G))-\frac{1}{|Q|}H_{\mu}\left(\bigvee_{g\in Q}g^{-1}\alpha | \mathcal{P}_{\mu}(G)\right)\right|<\epsilon.
		\end{equation}
	\end{thm}

	Repeat the statements of remark after \cite[Theorem 2.6]{HLY-Arxiv}, we can rewrite the formula (2.1) as 
	\begin{equation} \label{5}
		\int_X \left|\frac{1}{|Q|}\sum_{g\in Q}H_{\mu_{y}}(g^{-1}\alpha) - \frac{1}{|Q|}H_{\mu_{y}}\left(\bigvee_{g\in Q}g^{-1}\alpha\right)\right|d\mu(y)<\epsilon,
	\end{equation}
where $\mu=\int \mu_yd\mu(y)$ is the disintegration of $\mu$ with respect to $\mathcal{P}_{\mu}(G)$.

	The following result is from \cite[Theorem 0.4 (iii)]{DD}.
	
	\begin{lemma}\label{7}
		Let $(X,G)$ be a $G$-tds and $\mu \in \mathcal{M}^{e}(X, G)$. Let $\mu=\int \mu_yd\mu(y)$ be the disintegration of $\mu$ with respect to $\mathcal{P}_{\mu}(G)$ and $\lambda=\mu \times_{\mathcal{P}_{\mu}(G)} \mu$. Let $\pi:X \times X \rightarrow X $ be the projection to the first coordinate. Then  $$P_{\lambda}(G)=\pi^{-1}(\mathcal{P}_{\mu}(G)),$$
		 under ignoring $\lambda$-null sets.
	\end{lemma}

	
	
	

    \section{ Proof  of Theorem \ref{Thm-A}}
    In this section, we prove Theorem \ref{Thm-A}. Let us begin with the following lemma.
    \begin{lemma}\label{lem-1}Given $I,J\in\mathbb{N}$, we let $\mathcal{I}=\{1,2\ldots,I\}$ and $\mathcal{J}=\{1,2,\ldots,J\}$. For any  $\epsilon>0$ and $a_i^{(j)}\in [0,1]$, $i\in\mathcal{I}$, $ j\in\mathcal{J}$ with $\sum_{i=1}^Ia_i^{(j)}=1, j\in\mathcal{J}$, there exists $\delta>0$ such that the following holds: Let $(X,\mathcal{B},\mu)$ be a probability space and $\alpha^{(j)}=\{A_1^{(j)},A_2^{(j)},\cdots,A_I^{(j)}\}$, $j\in\mathcal{J}$ measurable partitions of $X$ with $\mu(A_i^{(j)})=a_i^{(j)}$  for each $i\in\mathcal{I}$, $j\in\mathcal{J}$. If $$\sum_{j=1}^JH_\mu(\alpha^{(j)})-H_\mu\left(\bigvee_{j=1}^J\alpha^{(j)}\right)<\delta,$$ then $$|\mu(\cap_{j=1}^JA_{i_j}^{(j)})-\Pi_{j=1}^Ja_{i_j}^{(j)}|\le \epsilon$$
    	for all $(i_1,i_2,\cdots,i_J)\in\mathcal{I}^J$.
\end{lemma}
\begin{proof}Define a continuous function
	$$F((t_{i_1,i_2,\cdots,i_J})_{i_1,i_2,\cdots,i_J\in \mathcal{I}})=-\sum_{i_1,i_2,\cdots,i_J\in \mathcal{I}}t_{i_1,i_2,\cdots,i_J}\log t_{i_1,i_2,\cdots,i_J}$$
	on a compact set
	\begin{align*}
	\Lambda=\{(t_{i_1,i_2,\cdots,i_J})_{i_1,i_2,\cdots,i_J\in \mathcal{I}}:t_{i_1,i_2,\cdots,i_J}\in[0,1],
	\sum_{i_{j}=i}t_{i_1,i_2,\cdots,i_J}=a_i^{(j)},i\in\mathcal{I},j\in\mathcal{J}\}.
	\end{align*}	
	 One has the following fact.
\begin{fact}\label{f1}
	For any $(t_{i_1,i_2,\cdots,i_J})_{i_1,i_2,\cdots,i_J\in \mathcal{I}}\in \Lambda$,
$$F((t_{i_1,i_2,\cdots,i_J})_{i_1,i_2,\cdots,i_J\in \mathcal{I}})=-\sum_{j=1}^J\sum_{i=1}^Ia_i^{(j)}\log a_i^{(j)},$$ implies
$$t_{i_1,i_2,\cdots,i_J}=\Pi_{j=1}^Ja_{i_j}^{(j)}\text{ for each }i_1,i_2,\cdots,i_J\in\mathcal{I}.$$
\end{fact}
\begin{proof}[Proof of Fact 1]
	Let $\phi(x)=\left\{
\begin{array}{lr}
-x\log x,                 & x>0\\
0,                  & x=0\\
\end{array}
\right. $.
Then by the concavity of $\phi(x)$, one has for any $i_J\in\mathcal{I}$,
\begin{equation}\label{44}
\begin{split}
-\sum_{i_1,i_2,\cdots,i_{J-1}\in \mathcal{I}}t_{i_1,i_2,\ldots,i_J}\log{\frac{t_{i_1,i_2,\ldots,i_J}}{\Pi_{j=1}^{J-1}a_{i_j}^{(j)}}}=&\sum_{i_1,i_2,\cdots,i_{J-1}\in \mathcal{I}}\Pi_{j=1}^{J-1}a_{i_j}^{(j)}\phi\lk({\frac{t_{i_1,i_2,\ldots,i_J}}{\Pi_{j=1}^{J-1}a_{i_j}^{(j)}}}\re)\\
\leq&\phi({\sum_{i_1,i_2,\cdots,i_{J-1}\in \mathcal{I}}{t_{i_1,i_2,\ldots,i_J}}})\\
=&\phi(a_{i_J}^{(J)})=-a_{i_J}^{(J)}\log{a_{i_J}^{(J)}}
\end{split}
\end{equation}
This implies that
\begin{align}
F((t_{i_1,i_2,\cdots,i_J})&_{i_1,i_2,\cdots,i_J\in \mathcal{I}})=-\sum_{i_J\in\mathcal{I}}\lk(\sum_{i_1,i_2,\cdots,i_{J-1}\in \mathcal{I}}t_{i_1,i_2,\ldots,i_J}\log{{t_{i_1,i_2,\ldots,i_J}}}\re)\notag\\
\leq&\sum_{i_J\in\mathcal{I}}\lk(-\sum_{i_1,i_2,\cdots,i_{J-1}\in \mathcal{I}}t_{i_1,i_2,\ldots,i_J}\log{{\Pi_{j=1}^{J-1}a_{i_j}^{(j)}}}-a_{i_J}^{(J)}\log{a_{i_J}^{(J)}}\re)\notag\\
=&-\sum_{s=1}^{J-1}\sum_{i_s\in \mathcal{I}}\sum_{i_1,\cdots,i_{s-1},i_{s+1},\cdots i_{J}\in \mathcal{I}}t_{i_1,i_2,\ldots,i_J}\log{{a_{i_s}^{(s)}}}-\sum_{i_J\in\mathcal{I}}a_{i_J}^{(J)}\log{a_{i_J}^{(J)}}\label{66}\\
=&-\sum_{s=1}^{J-1}\sum_{i_s\in \mathcal{I}}{a_{i_s}^{(s)}}\log{{a_{i_s}^{(s)}}}-\sum_{i_J\in\mathcal{I}}a_{i_J}^{(J)}\log{a_{i_J}^{(J)}}\notag\\
=&-\sum_{j=1}^J\sum_{i=1}^Ia_i^{(j)}\log a_i^{(j)}.\notag
\end{align}
Combining the assumption that $F((t_{i_1,i_2,\cdots,i_J})_{i_1,i_2,\cdots,i_J\in \mathcal{I}})=-\sum_{j=1}^J\sum_{i=1}^Ia_i^{(j)}\log a_i^{(j)}$, \eqref{44} and \eqref{66}, one has
$$\sum_{i_1,i_2,\cdots,i_{J-1}\in \mathcal{I}}\Pi_{j=1}^{J-1}a_{i_j}^{(j)}\phi\lk({\frac{t_{i_1,i_2,\ldots,i_J}}{\Pi_{j=1}^{J-1}a_{i_j}^{(j)}}}\re)=\phi(a_{i_J}^{(J)})$$
for each $i_J\in\mathcal{I}$. This implies that$$a_{i_J}^{(J)}=\frac{t_{i_1,i_2,\ldots,i_J}}{\Pi_{j=1}^{J-1}a_{i_j}^{(j)}}$$
for each $i_1,\cdots,i_J\in\mathcal{I}$. We finish the proof of Fact 1.
\end{proof}
Next we show that the lemma is true.  By contradiction, we may assume that there exist $\epsilon>0$ and 
	  partitions $\{\alpha_{n}^{(j)}=\{A_{1,n}^{(j)},A_{2,n}^{(j)},\cdots,A_{I,n}^{(j)}\}\}_{n\geq 1,j\in\mathcal{J}}$ of $X$ such that $\mu(A_{i,n}^{(j)})=a_i^{(j)}$  for each $i,j,n$,
	  $$\max_{i_1,\cdots,i_J\in\mathcal{I}}|\mu(\cap_{j=1}^JA_{i_j,n}^{(j)})-\Pi_{j=1}^Ja_{i_j}^{(j)}|\ge 2\epsilon,$$
	  for each $n$ and $$\lim_{n\to\infty}\sum_{j=1}^JH_\mu(\alpha_n^{(j)})-H_\mu\left(\bigvee_{j=1}^J\alpha^{(j)}_n\right)=0.$$
	  Put  $t_{i_1,i_2,\cdots,i_J}^{(n)}=\mu(\cap_{j=1}^JA_{i_j,n}^{(j)})$, then
	  $$(t^{(n)}_{i_1,i_2,\cdots,i_J})_{i_1,i_2,\cdots,i_J\in \mathcal{I}}\in \Lambda,$$
 $$\max_{i_1,\cdots,i_J\in\mathcal{I}}|t_{ i_1, i_2,\cdots, i_J}^{(n)}-\Pi_{j=1}^Ja_{ i_j}^{(j)}|\ge2 \epsilon,$$
 and 
$$\lim_{n\to \infty}F((t^{(n)}_{i_1,i_2,\cdots,i_J})_{i_1,i_2,\cdots,i_J\in \mathcal{I}})=-\sum_{j=1}^J\sum_{i}^{I}a_i^{(j)}\log a_i^{(j)}.$$
Since $\Lambda$ is  compact, passing by a subsequence of $((t^{(n)}_{i_1,i_2,\cdots,i_J})_{i_1,i_2,\cdots,i_J\in \mathcal{I}})_{n\ge 1}$ if necessary, we assume that $$\lim_{n\to\infty} t^{(n)}_{i_1,i_2,\cdots,i_J}=t_{i_1,i_2,\cdots,i_J},$$
	for each ${i_1,i_2,\cdots,i_J\in\mathcal{I}}$.  It is easy to see that 
	$$(t_{i_1,i_2,\cdots,i_J})_{i_1,i_2,\cdots,i_J\in \mathcal{I}}\in \Lambda,$$
	Then by the continuity of $F$, one has
	$$F((t_{i_1,i_2,\cdots,i_J})_{i_1,i_2,\cdots,i_J\in \mathcal{I}})=-\sum_{j=1}^J\sum_{i=1}^Ia_i^{(j)}\log a_i^{(j)},$$
and $$\max_{i_1,\cdots,i_J\in\mathcal{I}}|t_{ i_1, i_2,\cdots, i_J}-\Pi_{j=1}^Ja_{ i_j}^{(j)}|\ge2\epsilon,$$
	which contradicts Fact \ref{f1}. So Lemma \ref{lem-1} is true.
\end{proof}
With the help of Lemma \ref{lem-1}, we obtain the following result.
\begin{lemma}\label{lemma3}
	Let $(X,\mathcal{B}_X^\mu,\mu, G)$ be an ergodic $G$-mps.  For any $ f\in L^{2}(X,\B_{X}^\mu,\mu)$ and $ \epsilon>0$, there exist $K\in F(G)$ and $N\in \N$ such that for every  $Q\in F(G)$ with $(QQ^{-1} \backslash \{e_G\})\cap K=\emptyset$ and $|Q|\ge N$, 
	 $$\left\|\textbf{A}(Q,f)-\textbf{A}(Q,\mathbb{E}(f|\mathcal{P}_{\mu}(G)))\right\|_2<\epsilon.$$
\end{lemma} 

\begin{proof} If $h_\mu(G)=0$, then  $\mathcal{P}_\mu(G)=\mathcal{B}_X^\mu$ and hence $\mathbb{E}(f|\mathcal{P}_\mu(G))(x)=f(x)$ for $\mu$-a.e. $x\in X$ for any $f\in L^{2}(X,\B_{X}^\mu,\mu)$. So we have nothing to prove. 

Now we assume that $h_\mu(G)>0$. Let $\mu=\int \mu_yd\mu(y)$ be the disintegration of $\mu$ with respect to $\mathcal{P}_{\mu}(G)$. First, we show that the lemma is true for special functions $f=1_A\in L^2(X,\B_{X}^\mu,\mu)$, where $A$ satisfy that there exists $ B\in \mathcal{P}_{\mu}(G)$ and $q>0$ such that for $\mu$-a.e. $y\in B$, $\mu_{y}(A)=q$ and for $\mu$-a.e. $y\in B^{c}$, $\mu_{y}(A)=0$.

Fix $\epsilon>0$. By Lemma \ref{lem-1}, there exists $\delta\in(0,\epsilon^2/8)$ such that the following holds: for any $\rho\in\mathcal{M}(X)$,  if 
 finite measurable partitions $\beta_i=\{B_i,B_i^c\}$ of $(X,\mathcal{B}_X^\rho,\rho)$, $i=1,2$ satisfy  $$\rho(B_1)=\rho(B_2)=q$$ and $$H_{\rho}(\beta_1)+H_{\rho}(\beta_2)-H_{\rho}(\beta_1\vee\beta_2)<\delta,$$ 
then 
\begin{align}\label{e-3}
|\rho(B_1\cap B_2)-q^2|\le \epsilon^2/4.
\end{align}

Applying Theorem \ref{6} to $(X,\mathcal{B}_X^\mu,\mu,G)$ and the partition $\alpha=\{A,A^c\}$, we can find a finite subset $K$ of $G$ such that for any  $F\in F(G)$ with $(FF^{-1} \backslash \{e_G\})\cap K=\emptyset$, one has 
\begin{align}\label{e-4}	\int \left|\frac{1}{|F|}\sum_{g\in F}H_{\mu_{y}}(g^{-1}\alpha) - \frac{1}{|F|}H_{\mu_{y}}(\bigvee_{g\in F}g^{-1}\alpha)\right|d\mu(y)<\delta^{2}.\end{align}
Fix $N\in\mathbb{N}$ such that $1/{N}< \epsilon^2/4$ and we are going to show that  the subset $K$ of $X$ and $N\in\mathbb{N}$ are as required.
Let $Q$ be a finite subset of $G$ with
$|Q|>N$ and $ QQ^{-1}\setminus\{e_G\}\cap K=\emptyset.$ 
Then 
\begin{align}
&\left\|\textbf{A}(Q,f)-\textbf{A}(Q,\mathbb{E}(f|\mathcal{P}_{\mu}(G)))\right\|_2^2\notag\\
&=\frac{1}{|Q|^2}\sum_{g\in Q}\left\|f(gx)-\mathbb{E}(f|\mathcal{P}_{\mu}(G))(gx)\right\|_2^2+\notag\\
&\frac{1}{|Q|^2}\sum_{g\in Q,h\in Q, g\neq h}\lk(\int f(gx)f(hx)d\mu(x)-\int \mathbb{E}(f|\mathcal{P}_{\mu}(G)(gx)\mathbb{E}(f|\mathcal{P}_{\mu}(G)(hx)d\mu(x)\re)\notag\\
&=\frac{1}{|Q|^2}\sum_{g\in Q}\left\|f(gx)-\mathbb{E}(f|\mathcal{P}_{\mu}(G))(gx)\right\|_2^2+\label{22}\\
&\frac{1}{|Q|^2}\sum_{g\in Q,h\in Q, g\neq h}\int \int f(gx)f(hx)- \mathbb{E}(f|\mathcal{P}_{\mu}(G))(gx)\mathbb{E}(f|\mathcal{P}_{\mu}(G))(hx)d\mu_y(x)d\mu(y)\notag.
\end{align}
Clearly,
\begin{align}\label{e-1}
\frac{1}{|Q|^2}\sum_{g\in Q}\left\|f(gx)-\mathbb{E}(f|\mathcal{P}_{\mu}(G))(gx)\right\|_2^2\le \frac{2|Q|}{|Q|^2}< \epsilon^2/2.
\end{align}
So we only need to consider the second item in \eqref{22}, denoted by $\Pi$. By the properties of $f=1_A$, one has for $\mu$-a.e. $y\in X$, for $\mu_y$-a.e. $x\in X$
\begin{itemize}
	\item if $gx\in B^c$ or $hx\in B^c$, then  $$f(gx)f(hx)- \mathbb{E}(f|\mathcal{P}_{\mu}(G))(gx)\mathbb{E}(f|\mathcal{P}_{\mu}(G))(hx)=0;$$
	\item if $gx\in B$ and $hx\in B$, then
	$$f(gx)f(hx)- \mathbb{E}(f|\mathcal{P}_{\mu}(G))(gx)\mathbb{E}(f|\mathcal{P}_{\mu}(G))(hx)=1_A(gx)1_A(hx)-q^2.$$
\end{itemize}
Rewrite the second item in \eqref{22} by $$\Pi=\frac{1}{|Q|^2}\sum_{g,h\in Q, g\neq h}\int_{g^{-1}B\cap h^{-1}B}\mu_{y}(g^{-1}A\cap h^{-1}A)-q^2d\mu(y).$$
For each pair $(g,h)\in Q\times Q$ with $g\neq h$, we set $$D_1(g,h)=\{y\in X: gy\notin B\text{ or } hy\notin B\},$$
and $$D(g,h)=\{y\in X: |H_{\mu_{y}}(g^{-1}\alpha)+H_{\mu_{y}}(h^{-1}\alpha) - H_{\mu_{y}}(g^{-1}\alpha\vee h^{-1}\alpha)|<\delta\}.$$
Then
\begin{align}\label{DH}
g^{-1}B\cap h^{-1}B\subset \left(D(g,h)\setminus D_1(g,h)\right)\cup\left(X\setminus D(g,h)\right).
\end{align}
By \eqref{e-4}, \begin{align}\label{equi-0}
	\mu(D(g,h))>1-2\delta.
\end{align} 
For $\mu$-a.e $y\in D(g,h)\setminus D_1(g,h)$, one has 
$$\mu_y(g^{-1}A)=\mu_{gy}(A)=q\text{, }\mu_y(h^{-1}A)=\mu_{hy}(A)=q$$
and 
$$|H_{\mu_{y}}(g^{-1}\alpha)+H_{\mu_{y}}(h^{-1}\alpha) - H_{\mu_{y}}(g^{-1}\alpha\vee h^{-1}\alpha)|<\delta.$$
Combining \eqref{e-3} and the above fact, for $\mu$-a.e $y\in D(g,h)\setminus D_1(g,h)$,
\begin{align}\label{e-1-1}
|\mu_y(g^{-1}A\cap h^{-1} A)-q^2|\le \epsilon^2/4.	
\end{align} Hence we can estimate the second item $\Pi$ in \eqref{22} as follows.
\begin{equation}\label{33}
\begin{split}
|\Pi|&=\left|\frac{1}{|Q|^2}\sum_{g,h\in Q, g\neq h}\int_{g^{-1}B\cap h^{-1}B} \mu_{y}(g^{-1}A\cap h^{-1}A)-q^2d\mu(y)\right|\\
&\overset{\eqref{DH}}\leq\frac{1}{|Q|^2}\sum_{g\in Q,h\in Q, g\neq h}\int_{ D(g,h)\setminus D_1(g,h)}\left| \mu_{y}(g^{-1}A\cap h^{-1}A)-q^2\right|d\mu(y)+\\
&\ \ \  \frac{1}{|Q|^2}\sum_{g\in Q,h\in Q, g\neq h}\int_{X\setminus D(g,h)} \left|\mu_{y}(g^{-1}A\cap h^{-1}A)-q^2\right|d\mu(y)\\
&\overset{\eqref{e-1-1}}\le \frac{1}{|Q|^2}\sum_{g\in Q,h\in Q, g\neq h}\int_{ D(g,h)\setminus D_1(g,h)} \frac{\epsilon^2}{4} d\mu(y)+\\
& \ \ \ \frac{1}{|Q|^2}\sum_{g\in Q,h\in Q, g\neq h}\int_{X\setminus D(g,h)}1d\mu(y)\\
&\overset{\eqref{equi-0}}\le \epsilon^2/4+2\delta\le \epsilon^2/2.
\end{split}
\end{equation}
Combining \eqref{22}, \eqref{e-1}and \eqref{33}, one has 
$$\left\|\textbf{A}(Q,f)-\textbf{A}(Q,\mathbb{E}(f|\mathcal{P}_{\mu}(G)))\right\|_2<\epsilon.$$

Now we are going to show that the lemma is true for all $f\in L^2(X,\mathcal{B}_X^\mu,\mu)$. Fix $f\in L^2(X,\mathcal{B}_X^\mu,\mu)$ and $\epsilon>0$. 
By Lemma \ref{lemma2} and Lemma \ref{3}, there exist finitely many measurable subsets $A_1,\ldots,A_k$ of $X$ and real constants $a_1,\cdots,a_k$
  such that $$\|f-\sum_{i=1}^ka_i1_{A_i}\|_2<\epsilon/3$$
  and there exist $B_i\in\mathcal{P}_{\mu}(G)$ with $\mu(B_i)>0$ and $q_i>0$ such that for $\mu$-a.e. $y\in B_i$, $\mu_y(A_i)=q_i$  and for $\mu$-a.e. $y\in B_i^c$, $\mu_y(A_i)=0$ for each $i\in\{1,2,\ldots,k\}$. By the above proof, for each $i=1,2,\ldots,k$, there exists $K_i\in F(G)$ and $N_i\in\mathbb{N}$ such that for every $Q\in F(G)$ with $|Q|>N_i$ and $ (QQ^{-1}\setminus\{e_G\})\cap K_i=\emptyset$,
$$\left\|\textbf{A}(Q,1_{A_i})-\textbf{A}(Q,\mathbb{E}(1_{A_i}|\mathcal{P}_{\mu}(G)))\right\|_2<\epsilon/(3a_i(k+1)).$$ 
 Let $K=\cup_{i=1}^kK_i\in F(G)$ and $N=\max_{1\leq i\leq k}\{N_i\}$. Then for any $Q\in F(G)$ with $|Q|>N$ and $ QQ^{-1}\setminus\{e_G\}\cap K=\emptyset$, one has $|Q|>N_i$ and $ (QQ^{-1}\setminus\{e_G\})\cap K_i=\emptyset$ for each $i=1,2,\ldots,k$ and hence 
\begin{equation*}
\begin{split}
&\left\|\textbf{A}(Q,f)-\textbf{A}(Q,\mathbb{E}(f|\mathcal{P}_{\mu}(G)))\right\|_2\\
\leq&\|\textbf{A}(Q,f)-\textbf{A}(Q,\sum_{i=1}^ka_i1_{A_i})\|_2+\sum_{i=1}^ka_i\|\textbf{A}(Q,1_{A_i})-\textbf{A}(Q,\mathbb{E}(1_{A_i}|\mathcal{P}_{\mu}(G)))\|_2\\
&\qquad+\|\textbf{A}(Q,\mathbb{E}(f|\mathcal{P}_{\mu}(G)))-\textbf{A}(Q,\mathbb{E}(\sum_{i=1}^ka_i1_{A_i}|\mathcal{P}_{\mu}(G)))\|_2\\
\leq&\|f-\sum_{i=1}^ka_i1_{A_i}\|_2+k\cdot(\epsilon/(3(k+1)))+\|f-\sum_{i=1}^ka_i1_{A_i}\|_2\\
<&\epsilon/3+\epsilon/3+\epsilon/3=\epsilon.
\end{split}
\end{equation*}
This finishes the proof of Lemma \ref{lemma3}.
\end{proof}
To prove Theorem \ref{Thm-A}, we need the following observation.
\begin{lemma}\label{lem-2}Let $G$ be a group, $K,Q\in F(G)$ and $N\in\mathbb{N}$. If $|Q|\ge 2|K|N$, then there is $Q_1\subset Q$ such that $|Q_1|=N$ and $(Q_1Q_1^{-1}\setminus\{e_G\})\cap K=\emptyset.$ 
\end{lemma} 
\begin{proof}
	 Let $\hat Q_1$ be the largest subset of $Q$ such that if $g,h\in \hat Q_1$, $g\neq h$ then $h\notin (K\cup K^{-1})g$. 
	Then $Q\subset (K\cup K^{-1})\hat Q_1$ and $(\hat Q_1\hat Q_1^{-1}\setminus\{e_G\})\cap K=\emptyset$. This implies 
	$$|\hat Q_1|\ge \frac{|Q|}{|K\cup K^{-1}|}\ge \frac{|Q|}{2|K|}\ge N.$$
	Let $Q_1$ be a subset of $\hat Q_1$ with $|Q_1|=N$. Then $Q_1$ is as required.
\end{proof}
Now we are able to prove Theorem \ref{Thm-A}.
\begin{proof}[Proof of Theorem \ref{Thm-A}]Fix $f\in  L^2(X,\B_{X}^\mu,\mu)$ and $\epsilon>0$. By Lemma \ref{lemma3}, there is $K\in F(G)$ and $N\in\mathbb{N}$ such that   for every  $Q\in F(G)$ with $(QQ^{-1} \backslash \{e_G\})\cap K=\emptyset$ and $|Q|\ge N$, $$\left\|\textbf{A}(Q,f)-\textbf{A}(Q,\mathbb{E}(f|\mathcal{P}_{\mu}(G)))\right\|_2<\epsilon.$$
	

By using Lemma \ref{lem-2} several times,  each $G_n$ can be decomposed  into	
	$$Q_0^{(K,n)},Q_1^{(K,n)},Q_2^{(K,n)},\cdots, Q_{k_n}^{(K,n)}$$ such that
	\begin{itemize}
		\item[(1)] $Q_i^{(K,n)}\cap Q_j^{(K,n)}=\emptyset$ if $i\neq j$;
		\item[(2)] $|Q_0^{(K,n)}|<2|K|N$;
		\item[(3)] $|Q_i^{(K,n)}|=N$ if $i\neq 0$;
		\item[(4)] $(Q_i^{(K,n)}(Q_i^{(K,n)})^{-1}\setminus \{e_G\})\cap K=\emptyset$ if $i\neq 0$.
	\end{itemize}
Hence, 
	\begin{align*}&\left\|\textbf{A}(G_n,f)-\textbf{A}(G_n,\mathbb{E}(f|\mathcal{P}_{\mu}(G)))\right\|_2\\
	\le&\frac{1}{|G_n|}\Big(\sum_{g\in Q_0^{(K,n)}}\|f(gx)-\mathbb{E}(f|\mathcal{P}_\mu(G))(gx)\|_2+\\
	&\ \ \ \ \ \ \ \ \sum_{i=1}^{k_n} \|\sum_{g\in Q_i^{(K,n)}}f(gx)-\sum_{g\in  Q_i^{(K,n)}}\mathbb{E}(f|\mathcal{P}_{\mu}(G))(gx)\|_2\Big)\\
	\le&\frac{1}{|G_n|}\Big(\sum_{g\in Q_0^{(K,n)}}2\|f\|_2+\sum_{i=1}^{k_n} \epsilon N\Big)\\
	\le& \frac{1}{|G_n|}\Big(2N|K|\|f\|_2+|G_n|\epsilon\Big)\\
	\end{align*}
	Let $n\to\infty$. Then $\lim_{n\to\infty}|G_n|=\infty$ and hence one has
	\begin{align*}\limsup_{n\to\infty}\left\|\textbf{A}(G_n,f)-\textbf{A}(G_n,\mathbb{E}(f|\mathcal{P}_{\mu}(G)))\right\|_2\le \epsilon.
	\end{align*}
	Let $\epsilon\to 0$. Then
	\begin{align*}\lim_{n\to\infty}\left\|\textbf{A}(G_n,f)-\textbf{A}(G_n,\mathbb{E}(f|\mathcal{P}_{\mu}(G)))\right\|_2=0.
	\end{align*}
Therefore $\mathcal{P}_\mu(G)$ is a characteristic factor for $(G_n)_{n\ge 1}$.
\end{proof}

\section{Proof of Theorem \ref{thm1} and Theorem \ref{thm2}}
In this section, we introduce some lemma related  to weak convergence and prove Theorem \ref{thm1} and Theorem \ref{thm2}.
 \subsection{Weak convergence} A sequence $\{h_n\}_{n\ge 1}$ in $L^{2}(X,\B_{X}^\mu,\mu)$ converges weakly to $h\in L^{2}(X,\B_{X}^\mu,\mu)$ (write as $h_n\overset{w}\to h$),
if $\lim_{n\to\infty}\int  h_nfd\mu=\int hf d\mu$ for all $f \in L^{2}(X,\B_{X}^\mu,\mu)$.   

\begin{lemma}\label{lemmaA}
	Let $(X,\mathcal{B}_X^\mu,\mu,G)$ be a $G$-mps and $(G_n)_{n\ge 1}\subset F(G)$. Let $\mathcal{C}$ be a $G$-invariant sub-$\sigma$-algebra of $\mathcal{B}_X^\mu$ and $h\in L^{2}(X,\B_{X}^\mu,\mu)$. Assume that there exists  $P(h)\in L^{2}(X,\B_{X}^\mu,\mu)$ such that $$\textbf{A}(G_n,h)\overset{w}\rightarrow P(h).$$ 
	If 
	$$\lim_{n\to\infty}\left\|\textbf{A}(G_n,h)-\textbf{A}(G_n,\mathbb{E}(h|\mathcal{C}))\right\|_2=0,$$
	then
	$$\textbf{A}(G_n,\mathbb{E}(h|\mathcal{C}))\overset{w}\rightarrow  P(h).$$
			Especially, $P(h)\in L^2(X,\mathcal{C},\mu)$.
\end{lemma}
\begin{proof}For any $f\in L^{2}(X,\B_{X}^\mu,\mu)$, one has
	\begin{align*}&\limsup_{n\to\infty}\left|\int \left(\textbf{A}(G_n,h)-\textbf{A}(G_n,\mathbb{E}(h|\mathcal{C}))\right)fd\mu\right|\\
	\le&\limsup_{n\to\infty}\left\|\textbf{A}(G_n,h)-\textbf{A}(G_n,\mathbb{E}(h|\mathcal{C}))\right\|_{2}\|f\|_{2}\\
	=&0.
	\end{align*}
	This implies for any $f\in L^{2}(X,\B_{X}^\mu,\mu)$,
	\begin{align*}\lim_{n\to\infty}\int \textbf{A}(G_n,\mathbb{E}(h|\mathcal{C})) fd\mu&=\lim_{n\to\infty}\int \textbf{A}(G_n,h)fd\mu\\
	&=\int  P(h)fd\mu.
	\end{align*}
	Hence $\textbf{A}(G_n,\mathbb{E}(h|\mathcal{C}))\overset{w}\rightarrow  P(h)\in L^2(X,\mathcal{C},\mu).$ We end the proof of Lemma \ref{lemmaA}.
\end{proof}

\begin{lemma}\label{lemma1}
	Let $(X,\mathcal{B}_X^\mu,\mu,G)$ be a $G$-mps, $(G_n)_{n\ge 1}\subset F(G)$ and $h\in L^{2}(X,\B_{X}^\mu,\mu)$. If $\textbf{A}(G_n,h)\overset{w}\rightarrow P(h)\in L^{2}(X,\B_{X}^\mu,\mu)$, then we have the followings.
	\begin{itemize}
		\item[(1)] If there exists $M>0$ such that $h(x)\le M$ for $\mu$-a.e. $x\in X$, then\\ $\limsup_{n\to\infty}\textbf{A}(G_n,h)(x)\ge P(h)(x)$  for $\mu$-a.e. $x\in X$;
		\item[(2)] If  $h(x)>0$  for $\mu$-a.e. $x\in X $, then $P(h)(x)>0$ for $\mu$-a.e. $x\in X$. 
	\end{itemize}
\end{lemma}
\begin{proof} 	
	(1)  Let $D=\{x\in X:\limsup_{n\to\infty} \textbf{A}(G_n,h)(x)< P(h)(x)\}$. If $\mu(D)>0$, then
	\begin{align*}\limsup_{n\to\infty}\int \textbf{A}(G_n,h)(x)1_D(x)d\mu(x)\le& \int \left(\limsup_{n\to\infty}\textbf{A}(G_n,h)(x)\right)1_D(x)d\mu(x)\\
	<&\int P(h)(x)1_D(x)d\mu(x).
	\end{align*}	
	This is impossible since $\textbf{A}(G_n,h)\overset{w}\rightarrow P(h)$. Hence $\mu(D)=0$ and then (1) is valid.   
	
	(2) 	Let $B=\{x:P(h)(x)\le 0\}$. If $\mu(B)>0$, by the assumption $h(x)>0$  for $\mu$-a.e. $x\in X $, there exists $ \epsilon >0$ such that $\mu(C_{\epsilon})<\mu(B)$, where $C_{\epsilon}=\{x:h(x)< \epsilon\}$. For all $ g\in G$, $$\int h(gx)1_{B}(x)d\mu(x) \ge \int_{ B\setminus (g^{-1}C_\epsilon)} h(gx)d\mu(x) \ge\epsilon \mu(B\setminus (g^{-1}C_\epsilon))\ge \epsilon (\mu(B)-\mu(C_{\epsilon})).$$ So $$\int \textbf{A}(G_n,h)1_{B}d\mu \ge \epsilon (\mu(B)-\mu(C_{\epsilon})).$$ 
	Hence 
	$$\int P(h)1_{B}d\mu=\lim_{n\to\infty}\int \textbf{A}(G_n,h)1_{B}d\mu\ge  \epsilon (\mu(B)-\mu(C_{\epsilon}))>0,$$
	which contradicts the fact that $\int P(h)1_{B}d\mu\leq 0$. So $\mu(B)=0$, that is, $P(h)>0$ for $\mu$-a.e. $x \in X$.
\end{proof}

\subsection{Proof of Theorem \ref{thm1} and Theorem \ref{thm2}}    
   With the help of the results in Section 4.1, we are able to prove Theorem \ref{thm1}.
    \begin{proof}[Proof of Theorem \ref{thm1}]
    		Let $(X,G)$ be a $G$-tds, $(G_n)_{n\ge 1}\subset F(G)$ and $\mu\in\M^e(X,G)$. Assume that  $h_\mu(G)>0$ and $\lim_{n\to\infty}|G_n|=\infty$. Let  $\mu=\int \mu_yd\mu(y)$ be the disintegration of $\mu$ with respect to $\mathcal{P}_{\mu}(G)$. We are going to show that
    		for $\mu$-a.e. $x\in X$, there is $\delta_x>0$ such that  $\mu_x\times\mu_x(W_{(G_n)_{n\ge 1},\delta_x})=1$.
    		
    	Let $\lambda=\mu \times_{\mathcal{P}_\mu(G)} \mu$ and $h(x_1,x_2)=d(x_1,x_2)$. By Lemma \ref{3}, $\lambda(\Delta_{X})=0$, where $\Delta_{X}=\{(x,x): x\in X\}$, which implies $h(x_1,x_2)>0$   for $\lambda$-a.e. $(x_1,x_2)\in X \times X$.  Since $(\textbf{A}(G_n,h))_{n\ge 1}$ are  bounded in $L^{2}(X\times X,\B_{X\times X}^{\lambda},\lambda)$, there is a subsequence $(G_{n_k})_{k\ge1}$ of $(G_n)_{n\ge 1}$ such that 
    	\begin{equation}\label{11}
    	  	\textbf{A}(G_{n_k},h)\overset{w}\rightarrow P(h)\in L^2(X\times X,\B_{X\times X}^{\lambda},\lambda).
    	\end{equation}
    	 By Theorem \ref{Thm-A}, one has \begin{equation}\label{88}
    	\lim_{n\to\infty}\left\|\textbf{A}(G_n,h)-\textbf{A}(G_n,\mathbb{E}(h|\mathcal{P}_\lambda(G)))\right\|_2=0.
    	 \end{equation}
    	 Appling Lemma \ref{lemmaA} to \eqref{11} and \eqref{88}, we obtain that 
    	$$	P(h)\in L^2(X\times X,\mathcal{P}_\lambda(G),\lambda).$$
    	Moreover, accroding to Lemma \ref{7}, \begin{equation*}
    	\mathcal{P}_\lambda(G)=\pi^{-1}(\mathcal{P}_\mu(G)),
    	\end{equation*}
      where $\pi:X \times X \rightarrow X $ be the projection to the first coordinate.
    Hence one has
    	  \begin{equation*}
    	P(h)\in L^2(X\times X,\pi^{-1}(\mathcal{P}_\mu(G)),\lambda).
    	\end{equation*}
    	By Lemma \ref{lemma1} (2), one has $P(h)(x_1,x_2)>0$ for $\lambda$-a.e. $(x_1,x_2)\in X \times X$ and hence by Lemma \ref{3}
    	 for $\mu$-a.e. $x\in X$,  
    	$$\delta_x:=\frac{1}{2}\int P(h)d\mu_x\times\mu_x>0.$$  
    	As $P(h)\in L^2(X\times X,\pi^{-1}(\mathcal{P}_\mu(G)),\lambda)$, it follows that  for $\lambda$-a.e. $(x_1,x_2)\in X\times X$,
    	\begin{equation*}
    	P(h)(x_1,x_2)=\mathbb{E}(P(h)|\pi^{-1}(\mathcal{P}_\mu(G)))(x_1,x_2).
    	\end{equation*}
    Then we deduce that for	 $\mu$-a.e. $x \in X$, for  $\mu_x\times\mu_x$-a.e. $(x_1,x_2)\in X\times X$,
    	\begin{equation}\label{55}
    P(h)(x_1,x_2)=\int P(h)d\mu_x\times\mu_x.
    \end{equation}
    By Lemma \ref{lemma1} (1), for  $\mu$-a.e. $x \in X$, for  $\mu_x\times\mu_x$-a.e. $(x_1,x_2)\in X \times X$,
    	$$\limsup_{k\to\infty}\frac{1}{|G_{n_k}|}\sum_{g\in G_{n_k}}d(gx_1,gx_2)\ge P(h)(x_1,x_2)\overset{\eqref{55}}=\int P(h)d\mu_x\times\mu_x> \delta_x.$$
    	 Hence for $\mu$-a.e. $x \in X$,
    	$$\mu_x\times\mu_x(W_{(G_{n})_{n\ge 1},\delta_x})=1.$$
    	We finish the proof of Theorem \ref{thm1}.
    \end{proof}




Now we prove Theorem \ref{thm2}. 	
\begin{proof}[Proof of Theorem \ref{thm2}]Put
	$$Prox((G_n)_{n\ge 1})=\{(x,y)\in X\times X:\liminf_{n\to\infty}\frac{1}{|G_n|}\sum_{g\in G_n}d(gx,gy)=0\}.$$
For each $x\in X$,
$$W_{S,(G_n)_{n\ge 1}}(x)\times W_{S,(G_n)_{n\ge 1}}(x)\subset Prox((G_n)_{n\ge 1}).$$ By condition \eqref{Condition} for $\mu$-a.e. $x\in X$,
	 $$\overline{Prox((G_n)_{n\ge 1})\cap (supp(\mu_x)\times supp(\mu_x))}=supp(\mu_x)\times supp(\mu_x).$$
	  Hence for $\mu$-a.e. $x\in X$, ${Prox((G_n)_{n\ge 1})\cap (supp(\mu_x)\times supp(\mu_x))}$ is a dense $G_\delta$ subset of $supp(\mu_x)\times supp(\mu_x)$.

	For $\mu$-a.e. $x\in X$, let $\delta_x>0$ be as in Theorem \ref{thm1}. Then for $\mu$-a.e. $x\in X$, $\mu_x\times\mu_x(W_{G_{n},\delta_x})=1.$ By Lemma \ref{3}, we know for $\mu$-a.e. $x\in X$, $W_{G_{n},\delta_x}\cap  (supp(\mu_x)\times supp(\mu_x))$ is a dense $G_\delta$ subset of $supp(\mu_x)\times supp(\mu_x)$.
	
	Therefore for $\mu$-a.e. $x\in X$,
	$$W_{G_{n},\delta_x}\cap Prox((G_n)_{n\ge 1})\cap (supp(\mu_x)\times supp(\mu_x))$$
	is a dense $G_\delta$ subset of $supp(\mu_x)\times supp(\mu_x)$. By Theorem \ref{Mycielski} and Lemma \ref{3}, for $\mu$-a.e. $x\in X$, there is a dense Mycielski $((G_n)_{n\ge 1},\delta_x)$-mean scrambled subset $K_x$ of $supp(\mu_x)\times supp(\mu_x)$. This ends the proof of Theorem \ref{thm2}.
	\end{proof}

\section*{Acknowledgement}
 C. Liu is supported in part by NNSF of  China (12090012). R. Xiao is supported by NNSF of China (11971455, 12031019, 12090012). L. Xu
 is supported in part by NNSF of China (12090012,11871188) and the USTC Research Funds of the Double First-Class Initiative.

\end{document}